\documentclass{amsart}
\usepackage[latin1]{inputenc}
\usepackage{latexsym}
\usepackage{amsfonts}
\usepackage{amssymb}
\usepackage{color}
\usepackage{dsfont}
\usepackage{varioref}
\usepackage{epsfig}
\usepackage{xy}
\usepackage[active]{srcltx}
\usepackage{array}
\RequirePackage{graphics,epsfig,psfrag}
\usepackage{color}
\usepackage{colordvi,fancyhdr}
\usepackage{pdfpages}
\definecolor{Gris}{cmyk}{0.1,0.1,0.1,.75}
\usepackage[colorlinks=true,citecolor=Gris,linkcolor=Gris,filecolor=Gris,urlcolor=Gris]{hyperref}
\def\pn{\medskip\par\noindent}

\def\frac#1#2{{\textstyle{{#1} \overwithdelims.. {#2}}}}
\def\Frac#1#2{{\displaystyle{{#1} \overwithdelims.. {#2}}}}

\newcommand{\Pf}{\noindent{\em Proof}. }

\newcommand{\EPf}{\hbox{}\hfill$\Box$\vspace{.5cm}}
\def\H{\mathrm{H}}
\def\T{\mathrm{T}}
\def\lk{\mathrm{lk}\,}

\newtheorem{thm}{Theorem}[section]
\newtheorem*{thm*}{Theorem}
\newtheorem{defi}[thm]{Definition}
\newtheorem{prop}[thm]{Proposition}
\newtheorem{lemma}[thm]{Lemma}
\newtheorem{cor}[thm]{Corollary}
{\theoremstyle{definition}
\newtheorem{rem}[thm]{Remark}}{\theoremstyle{definition}
\newtheorem{exa}[thm]{Example}}
\newtheorem{remark}[thm]{Remark}

\newcommand{\ZZ}{{\mathbb Z}}

\newcommand{\CC}{{\mathbb C}}
\newcommand{\RR}{{\mathbb R}}

\newcommand{\LL}{{\mathcal L}}

\def\pent#1#2{\pe{\frac{#1}{#2}}}

\def\pe#1{{\left \lbrack #1 \right \rbrack}} 
\def\Mod#1{\,(\hbox{\rm mod}\,#1)}
\def\[#1\]{\begin{equation} #1 \end{equation}}

\begin{document}
\title{On  the lexicographic degree of two-bridge knots}
\date{\today}
\author{Erwan Brugall\'e}
\author{Pierre-Vincent Koseleff}
\author{Daniel Pecker}

\begin{abstract}
We study the degree of polynomial representations of knots.
We obtain the lexicographic degree for two-bridge torus knots and generalized twist knots.
The proof uses the braid theoretical method developed by Orevkov to study real plane curves,
combined with previous results from \cite{KP4} and \cite{BKP1}.
We also give a sharp lower
bound for the lexicographic degree of
any knot, using real polynomial curves properties.
\end{abstract}
\maketitle
\section{Introduction}
It is known that every knot in $ {\bf S}^3$ can be represented as the closure of the
image of a polynomial embedding
 $\RR \to \RR ^3 \subset {\bf S}_3 $, see \cite{Sh,Va}.
In these early papers on the subject, only a few specific examples were given.
The two-bridge torus knots form the first infinite
family of knots for which a polynomial representation
was explicitly given, see \cite{RS,KP1,KP3,KP4}.

\begin{exa}[Harmonic knots (\cite{KP3, KP6})] The harmonic knot $\H(a,b,c)$ is
the polynomial knot  parametrized by
$\left(T_a(t),T_b(t),T_c(t) \right) $,
where $T_n$ is the classical
Chebyshev polynomial $T_n(\cos t)= \cos nt$ and $a,b,c$ are pairwise coprime integers.
\begin{figure}[!ht]
\begin{center}
\begin{tabular}{ccc}
{\scalebox{.8}{\includegraphics{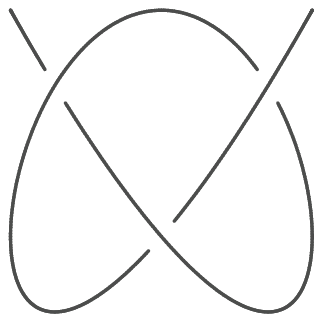}}}&&
{\scalebox{.8}{\includegraphics{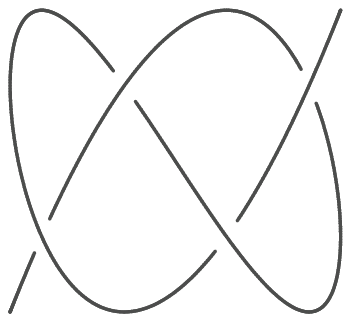}}}\\
$3_1=\H(3,4,5)$&&$4_1=\H(3,5,7)$
\end{tabular}
\end{center}
\caption{Two harmonic knots}
\label{h357}
\end{figure}
The two knot diagrams depicted in Figure \ref{h357} are the
\emph{$xy$-diagrams} of the knots $\H(3,4,5)$ and $\H(3,5,7)$,
 corresponding to the projection $(x,y,z)\mapsto(x,y)$.
\end{exa}
The \emph{multidegree} of a polynomial map $\gamma:\RR\to\RR^n, t\mapsto
(P_i(t))$ is the $n$-tuple $(\deg(P_i))$. The \emph{lexicographic
degree} of a knot $K$ is the minimal multidegree, for
the lexicographic order, of a polynomial knot whose closure in
${\bf  S}^3$ is isotopic to $K$.
The unknot has
lexicographic degree $(-\infty,-\infty,1)$, and one sees
easily that the  lexicographic degree of any other knot is
$(a,b,c)$ with $3\le a<b<c$. Given a knot, it is in general a
difficult problem to determine its lexicographic degree.
In particular, the corresponding diagram might not have the minimal number of crossings.
\medskip\par
In this paper we investigate the case of \emph{two-bridge} knots,
that are the knots for which $a=3$, see \cite{Cr,KP4,BKP1}. We focus
in particular on   two-bridge torus knots and
generalized twist knots,
that we respectively denote by $C(m)$ with $m$ odd, and
$C(m,n)$ with $mn$ positive and even (see Figure \ref{fig:2b0}). Note that $C(-m)$
and $C(-m,-n)$ are the mirror images of $C(m)$ and $C(m,n)$.
\begin{figure}[!ht]
\begin{center}
\psfrag{a}{$m$}\psfrag{b}{$n$}
\begin{tabular}{ccc}
{\scalebox{.8}{\includegraphics{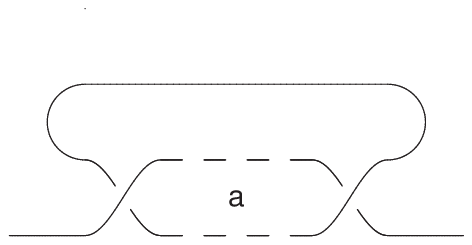}}}&&
{\scalebox{.8}{\includegraphics{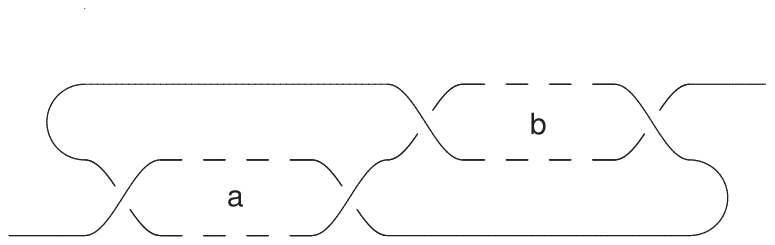}}}
\\
$C(m)$&&$C(m,n)$
\end{tabular}
\end{center}
\caption{Two-bridge torus knots and generalized twist knots}
\label{fig:2b0}
\end{figure}
\pn
Figure \ref{fig:2b1} shows  typical examples, $C(5)$ and $C(4,3).$
\begin{figure}[!ht]
\begin{center}
\begin{tabular}{ccc}
{\scalebox{.8}{\includegraphics{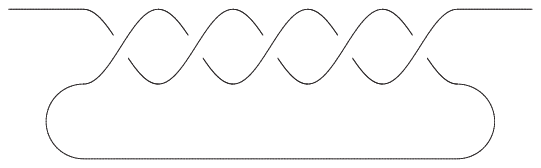}}}&&
{\scalebox{.8}{\includegraphics{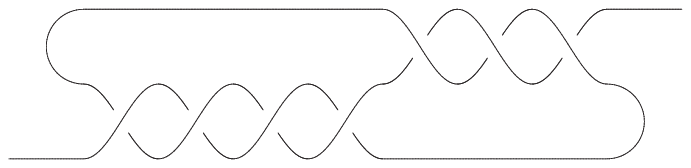}}}\\
$5_1=C(5)$&&$7_3=C(4,3)$
\end{tabular}
\end{center}
\caption{The knots $5_1$ and $7_3$}
\label{fig:2b1}
\end{figure}
\pn
Mishra claimed in \cite{Mi} that
the  lexicographic degree of the torus knot $C(m)$,
$m \not\equiv 2 \Mod 3 $, is $ (3, m+1, c)$.
Later on, it was proved in \cite{KP1} that this is false for every integer $m>5$.
At that stage, the  lexicographic degree of only a few knots was known,
namely the trefoil, the figure-eight knot, the torus knot $5_1$ \cite{Mi}, and
the torus knot $7_1$ \cite{KP1}.

For  every two-bridge knot of crossing number $N$,
a polynomial representation of degree $ (3,b,c)$,  with $b+c=3N$ and
$b \ge N+1$ is constructed in \cite{KP4}.
From that construction an upper bound for the  lexicographic
degree of two-bridge knots is immediately deduced.
The main result of this paper is that this upper bound is sharp for
two-bridge torus knots and for generalized twist knots.
\pn
{\bf Theorem \ref{thm:main}.}
{\em
Let $K$ be the two-bridge knot $C(m)$, or $C(m,n)$ with $mn>0$,
and let
$N$ be  its crossing number. Then the  lexicographic degree of $K$ is
$$ \Bigl(3, \pent{3N-1}2, \pent{3N}2+1 \Bigr).$$}
\pn
The proof  goes by the study of the plane
curve defined by the $xy$-projection of a polynomial knot, and  uses
 Orevkov's braid theoretical approach to study real plane curves (see
 \cite{O1} or Section \ref{sec:pcld}).
Like in the case of knots, we define the  lexicographic
degree of a long knot diagram as the minimal multidegree, for
the lexicographic order, of a polynomial knot whose $xy$-diagram is
isotopic in $\RR ^2$ to $D$.

The main idea in the proof of  Theorem \ref{thm:main}
is as follows. Suppose that $D_1$ is a trigonal long knot diagram
with
 lexicographic degree $(3,b,c).$
If $D_2$ is a trigonal diagram obtained from $D_1$ by an isotopy
 through trigonal diagrams
which never increases the number of crossings, then one may expect
that it
has  lexicographic degree
$(3,b',c')$ with $b'\le b.$
 If this was the case,
 then Theorem \ref{thm:main} would follow from
 the study of the alternating diagrams
 $C(m)$ or $C(m,n)$, combined with
the following theorem.
\pn
{\bf Theorem A.} (\cite{BKP1})
{\em Let $K$ be the two-bridge knot $C(m)$, or $C(m,n)$ with $mn>0$,
and let $D$ be a trigonal diagram of $K$.
Then
it is possible to transform $D$
into an alternating trigonal diagram,
so that all intermediate diagrams remain trigonal, and the number of crossings never increases.}
\pn
Unfortunately, one cannot ensure that the lexicographic degree of a
trigonal knot diagram does not increase during an isotopy, even if this isotopy
never increases the number of crossings. For this to be true, one needs
to enlarge the set of objects we are studying, namely we have to consider
\emph{real pseudoholomorphic curves} (see
Corollary \ref{cor:isotopy}). Note that a real algebraic curve
is real pseudoholomorphic, however the converse is not true in general.
\pn
We conclude the paper by providing a lower bound for the
lexicographic degree  of every non-trivial knot. We also give the
lexicographic degree of some infinite families of knots,
showing in particular that this lower bound is sharp.
\pn
{\bf Theorem \ref{thm:least}.}
{\em
Let $K$ be a knot of crossing number $N > 0.$
Then the  lexicographic degree of $K$ is at least $(3, N+1, 2N-1).$
For every integer $ N\not\equiv 2 \, \Mod 3, N \ge 3,$ there exists
a knot of crossing number $N$ and  lexicographic degree $ (3, N+1, 2N-1)$.}
\medskip

The paper is organized as follows.
We recall Orevkov's braid theoretical method in Section
\ref{sec:pcld}. To avoid technical difficulties  coming from
an excess of generality, we restrict ourselves to the case of
real
curves of  bidegree $(3,b)$, whereas
the whole section generalizes to all bidegrees.
We prove Theorem \ref{thm:main} in Section \ref{rph}.
To keep the exposition as light as possible
we keep working with real algebraic curves
as long as pseudoholomorphic curves are not needed. In particular we
prove the lower bound of Theorem \ref{thm:main}
in the particular case of the diagrams $C(m)$ and $C(m,n)$
for real algebraic curves (Proposition \ref{prop:lower bound}).
We introduce real pseudoholomorphic curves only in Section \ref{sec:general}
to deduce the general lower bound of Theorem \ref{thm:main} from
Proposition \ref{prop:lower bound}
(Proposition \ref{prop:lower}).
In Section \ref{easy}, we first obtain a sharp lower bound for alternating 
diagrams.
The proof is elementary, it is based on polynomial plane curves properties.
As a consequence, we deduce Theorem \ref{thm:least} and an upper 
bound for the crossing number of alternating knots of polynomial degree $d$.
\section{The braid of a real algebraic curve in $\CC^2$}
\label{sec:pcld}
Here we recall basic facts about the braid theoretical approach
developed by Orevkov  to study real
algebraic curves in $\CC^2$.
\subsection{Link associated to a real algebraic curve}\label{braid}
Given  $C\subset\CC^2$ a real algebraic curve,
we denote by $\RR C$ the \emph{real part} of $C$, i.e.
$\RR C= C\cap \RR^2$.
\def\im{\mathrm{Im}\,}
Let us fix an orientation preserving diffeomorphism $\Phi:\RR^4\to
\{(x,y)\in\CC^2\ | \ \im(x)>0\}$,
and let us denote by $B_r$
the image by $\Phi$ of the $4$-ball of radius $r$,
and by $S_r=\partial B_r$ the image by $\Phi$ of the $3$-sphere
of radius $r$.
If $C\subset\CC^2$ is a real
algebraic curve, then all links $S_r\cap C$ (resp. all
 surfaces $B_r\cap C$) are isotopic if $r$ is large enough.
\begin{defi}
The link $L_C=S_r\cap C$, for $r$ large enough, is called the link
associated to the real algebraic curve $C$.
\end{defi}
We denote $U_C=B_r\cap C$, and we orient $L_C$
 as the boundary of $U_C$.
We will use the following standard proposition in the proof of
Theorem \ref{thm:main}.
The linking number of two oriented links $L_1$ and $L_2$ is denoted by
$\lk(L_1,L_2)$; the algebraic intersection number of two smooth surfaces
$U_1$ and $U_2$ in $B_r$ intersecting transversally and in $B_r\setminus S_r$
is denoted by $U_1 \cdot U_2$.
\begin{prop}\label{inters}
If $L_1$ and $L_2$ are two sublinks of $L_C$ which bound two components
$U_1$ and $U_2$ of $N_C$, then
$$\lk(L_1,L_2)=U_1 \cdot U_2. $$
\end{prop}
It turns out that in some situations,
one can read directly on the real part $\RR C$ an
expression of the oriented link $L_C$ as a closed braid.
For the sake of simplicity,
we  restrict ourselves here to the very special case of trigonal
rational curves.
For the general case, we refer to \cite{O1}.
In what follows,
we are interested in curves in $\CC^2$ up to isotopies respecting the fibration of $\CC^2$ by vertical lines.
\begin{defi}
An isotopy on a curve $C \subset \CC^2$ is called a 
{\em $\mathcal L$-isotopy} if it commutes with the projection
$\pi :  \CC^2 \to  \CC,  \, (x,y) \mapsto x$.
\end{defi}

Let us consider a polynomial map
$$\begin{array}{cccc}
\gamma: & \CC &\longrightarrow & \CC^2
\\ & t &\longmapsto & (P(t), Q(t))
\end{array} $$
where $P(t)$ and $Q(t)$ are two real polynomials of degrees $3$
and $b\ge 1$. Replacing $P(t)$ with $-P(t)$ if necessary, we may assume that $P(t)$ is positive for $t$
large enough. We may also assume that $b=3k-1$ or $b=3k-2$. If
$b=3k$, then we can replace the map $\gamma$ by the map
$t \to (P(t), Q(t) -\alpha P^k(t))$.
In the target space $\CC^2$,
this corresponds to performing the real change of coordinates
$y=y-\alpha x^k$ and  does not affect the topology of
the real algebraic curve $C=\gamma(\CC)$ in $\CC^2$.

We suppose in addition that the map $\gamma$ is generic enough so that:
\begin{itemize}
\item[---] $\gamma$  is an
immersion;
\item[---] if $\gamma(t_1)=\gamma(t_2)$ for $t_1\ne t_2$, then $\gamma'(t_1)\ne
\gamma'(t_2)$ and $\gamma^{-1}(\gamma(t_1))=\{t_1,t_2\}$.
\end{itemize}
In other words, the only singularities of the embedded algebraic curve
$C\subset \CC^2$ are nodes.
Since $C$ is  real, it has three kinds of nodes:
\begin{enumerate}
\item the intersection of two real branches of $C$ (i.e.
$t_1,t_2\in\RR$); such a real node is called a crossing,
see Figure \ref{ex nodes}a;
\item the intersection of two complex conjugated branches of $C$ (i.e.
$t_2=\overline{t_1}$); such a real node is called solitary node, see Figure \ref{ex nodes}b;
\item the intersection of two branches of $C$ which are
  neither real nor complex conjugated; all such nodes lie in
  $\CC^2\setminus \RR^2$ and come in pairs of complex conjugated nodes.
\end{enumerate}
\begin{figure}[!ht]
\begin{center}
\begin{tabular}{ccc}
\includegraphics[height=1.5cm, angle=0]{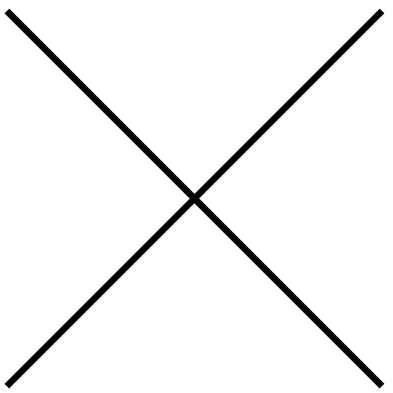}& \hspace{10ex} &
\includegraphics[height=1.5cm, angle=0]{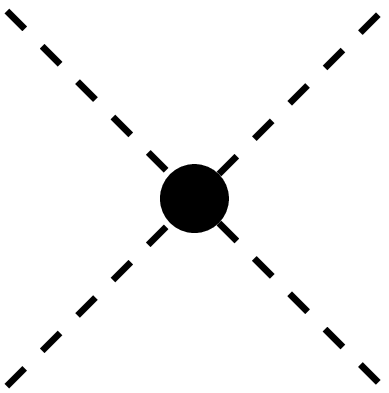}
\\ a) A crossing && b) A solitary node
\end{tabular}
\end{center}
\caption{}
\label{ex nodes}
\end{figure}
We denote by $N$ the number of crossings of $C$, by $\alpha$ the number of solitary nodes
of $C$ and by $2\beta$ its number of nodes lying in $\CC^2\setminus\RR^2$.
Hence the real part $\RR C$ is the union of $\gamma(\RR)$ together with the
$\alpha$ solitary nodes of $C$.
\begin{lemma}\label{lem:nodes}
Given a map $\gamma:\CC\to\CC^2$ as above, we have $b-1=N+\alpha+2\beta$.
\end{lemma}
\begin{proof}
The Newton polygon $T$ of the curve $\gamma(C)$ is the triangle with
vertices $(0,0),(0,3)$ and $(b,0)$, and the number of nodes of
$\gamma(C)$ is the number of points in $\ZZ^2$ contained in the interior
of $T$.
\end{proof}
\begin{exa}\label{ex:running ex}
Let $1/\sqrt 3 < \nu < 1$.
Using elementary elimination theory, one sees  that
the real part of the parametrized curve
$ \gamma: t \mapsto (T_3(t), T_4(t+\nu)) $
is depicted,  up to $\mathcal L$-isotopy, in Figure \ref{Ex1}a. It has two
crossings and one solitary node.
\def\imagetop#1{\vtop{\null\hbox{#1}}}
\begin{figure}[h]
\begin{center}
\begin{tabular}{ccc}
\imagetop{\scalebox{.25}{\includegraphics{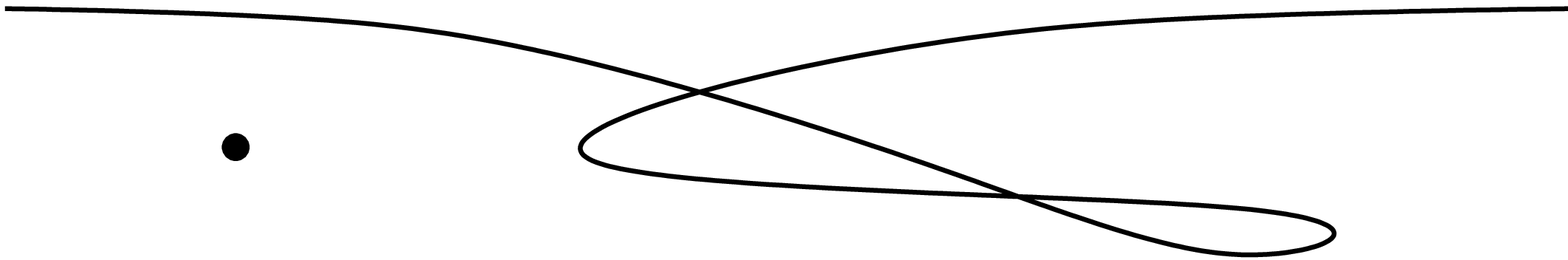}}}&&
\imagetop{\scalebox{.25}{\includegraphics{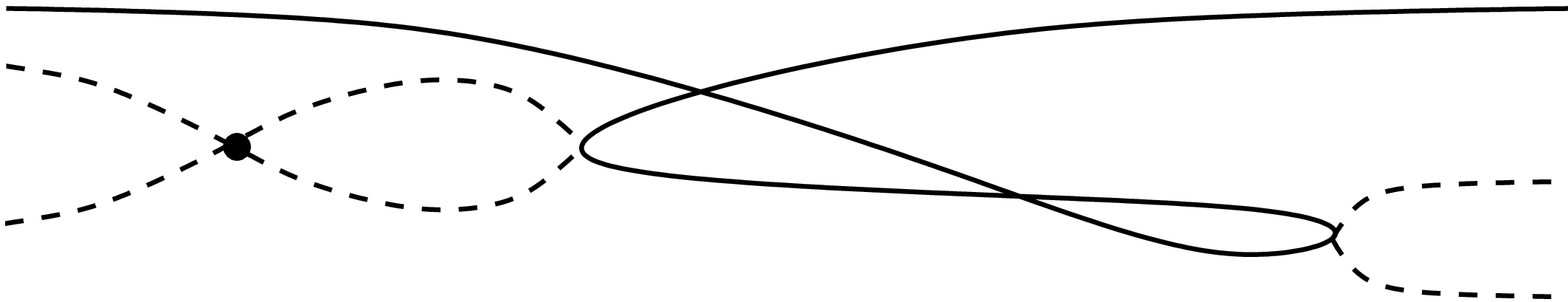}}}\\
\\ a) $\RR C$&&b) $A$
\end{tabular}
\end{center}
\caption{}
\label{Ex1}
\end{figure}
\end{exa}
Now let us consider the set $A=\pi^{-1}(\mathbb R)\cap C$. It
consists in the real part $\RR C$ together with
some pairs of complex conjugated arcs, meeting at critical points
of the map $\pi_{|\RR C}$ (i.e. at points of $\RR C$
which intersect  a vertical line non-transversely).

\begin{exa}
The set $A$ in the
case of the parametrization of Example \ref{ex:running ex} is depicted
in Figure \ref{Ex1}b. The dotted curves correspond to complex
conjugate arcs.
\end{exa}

The set $A$ is not a manifold because of critical points of the map
$\pi_{|\RR C}$. However there is a canonical way to perturb
$A$ to a $1$-dimensional manifold and to recover the link $L_C$.
Let $R_1, R_2$ and $\varepsilon$ be some positive real numbers, and
let
$$V_1=\pi^{-1}\left(\{x\in \CC\ | \ |x|\le R_1 \mbox{ and }
\im(x)>\varepsilon\}\right)\ \mbox{ and } \
V_2=\{y \in\CC \ | \ |y|\le
R_2\}.$$
 Then for $R_1$ and $R_2$ large enough, and
$\varepsilon$ small enough, the link
$C\cap \partial \left(V_1\times V_2\right)$ in
$\partial \left(V_1\times V_2\right)\simeq S^3$
 is precisely the link $L_C$, and the surface
$C\cap \left(V_1\times V_2\right)$ is precisely the surface $U_C$ in
$ V_1\times V_2\simeq B^4$. Moreover the link $C\cap \partial
 \left(V_1\times V_2\right)$ appears naturally as the
 closure of a braid $b_C$, that we describe in the next section.

\subsection{The link $L_C$ as a closed braid}\label{sec:closed braid}
Recall that the {\em group of braids with $3$-strings} is defined as
$$B_3=\langle \sigma_1, \sigma_{2}|
\sigma_1\sigma_{2}\sigma_1=\sigma_{2}\sigma_{1}\sigma_{2}\rangle.$$
This terminology comes from the fact that there exists a natural
geometric interpretation of $\sigma_1$, $\sigma_2$, and of the
relation $\sigma_1\sigma_{2}\sigma_1=\sigma_{2}\sigma_{1}\sigma_{2}$
as depicted in Figure \ref{fig:braid}. Note that a string of a
braid is
implicitly oriented from left to right. In particular if an oriented link $L$
is represented as the closure of a braid $b$, then the linking number of two
sublinks $L_1$ and $L_2$ of $L$ is  half the sum of all exponents of
$b$ corresponding to the crossings of $L_1$ and $L_2$.

\begin{figure}[h]
\begin{center}
\begin{tabular}{ccccc}
\includegraphics[height=1cm, angle=0]{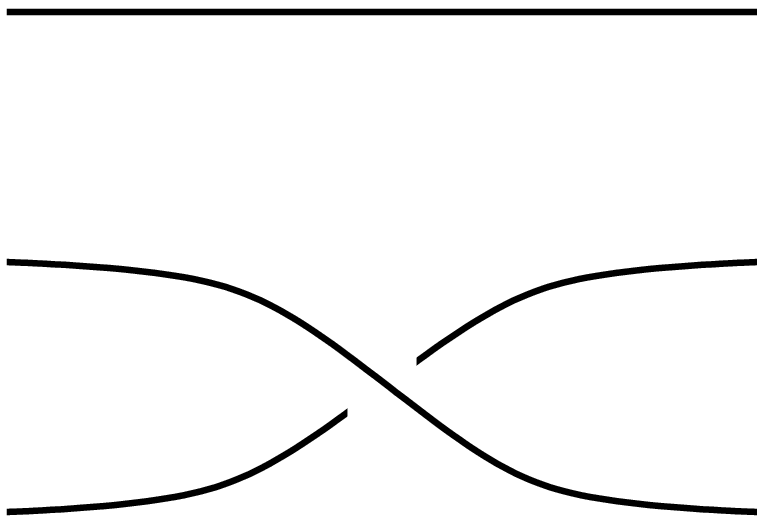}& \vspace{3ex} &
\includegraphics[height=1cm, angle=0]{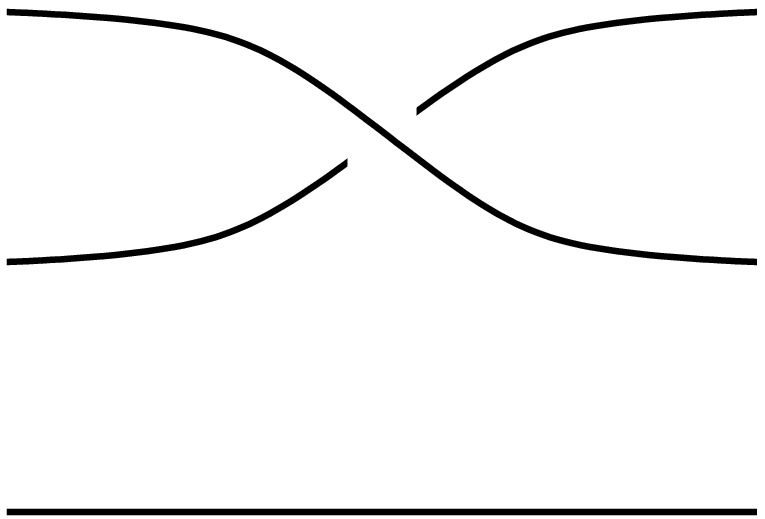}& \vspace{3ex} &
\includegraphics[height=1cm, angle=0]{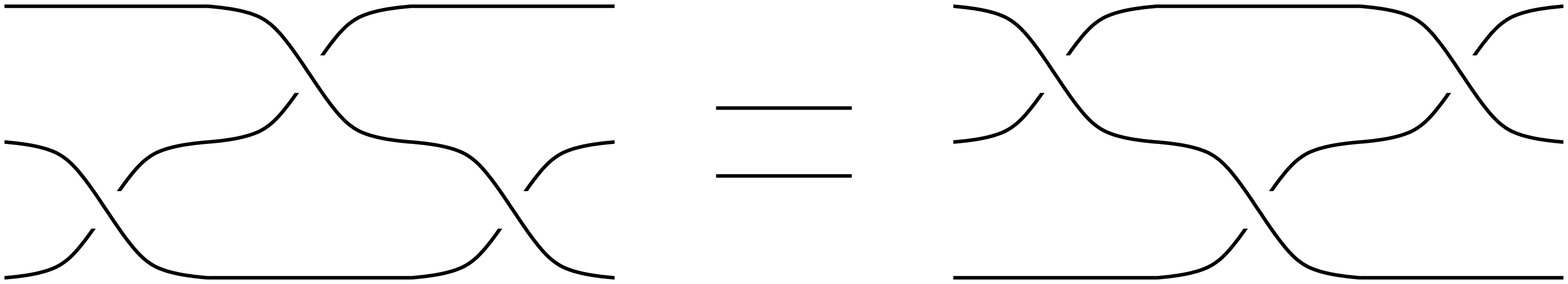}
\\ a)  $\sigma_1$ && b) $\sigma_2$ && c) $\sigma_1\sigma_{2}\sigma_1=\sigma_{2}\sigma_{1}\sigma_{2}$
\end{tabular}
\end{center}
\caption{}
\label{fig:braid}
\end{figure}

Let us  associate an \textit{$\mathcal L$-scheme}
to the real curve $C$.
Denote by $ l_1< l_2<\ldots< l_s$ the real vertical lines of
$\CC^2$ which are not  transversal to
$C$ (the order on the $l_i$'s is the one induced by the canonical orientation of
$\RR$). Recall that the degree of the polynomial $Q(t)$ is $b=3k-1$ or $b=3k-2$.
\begin{defi}
The $\mathcal L$-scheme $\LL_C$ realized by $C$ is given by the sequence
$c_1c_2\ldots c_{s+1}$ where
$c_i$  with $1\le i\le s$ is given by
\begin{itemize}
\item[---]  $c_i=\supset_j$ (resp. $c_i=\subset_j$)
if $1\le i\le s$ and
$l_i$ has an ordinary tangency point with $C$ which is a local
maximum (resp. minimum) of the map $\pi|_{\RR C}$;
\item[---] $c_i=\times_{j}$ (resp. $c_i=\bullet_{j}$) if  $i\le s$ and
$l_i$ passes through a
 crossing (resp. solitary node) of $\RR C$;
\end{itemize}
(in each case $j=1$ if the transverse intersection point of
$\RR l_i$ and $\RR C$ is below the second point of $\RR l_i\cap\RR C$,
and $j=2$ otherwise)

and $c_{s+1}$ is given by
\begin{itemize}
\item[---]  $c_{s+1}=\downarrow$ (resp. $\uparrow$)
if  $b=3k-1$, and $Q(t)$ is positive (resp. negative) for $t$ large
enough;
\item[---] $c_{s+1}=\vee$ (resp. $\land$)
if  $b=3k-2$, and $Q(t)$ is positive (resp. negative) for $t$ large
enough.
\end{itemize}
\end{defi}

\begin{exa}
The $\mathcal L$-scheme realized by the parametrization of Example
\ref{ex:running ex} is
$$\bullet_1 \subset_1\times_2\times_1\supset_1 \vee.$$
\end{exa}
\vspace{1ex}
Then replace the $\mathcal L$-scheme $c_1c_2\ldots c_{s+1}$ by
$c'_0c_1c_2\ldots c'_{s+1}$ where
\begin{itemize}
\item[---] $c'_0= \supset_2$ (resp. $\supset_1$) and $c'_{s+1}= \subset_1$
if  $c_{s+1}=\downarrow $ and
$k$ is even (resp. odd);

\item [---] $c'_0= \supset_1$ (resp. $\supset_2$) and $c'_{s+1}= \subset_2$
if   $c_{s+1}=\uparrow $ and
$k$ is even (resp. odd);

\item[---]  $c'_0= \supset_2$ (resp. $\supset_1$) and
$c'_{s+1}=\subset_{1}\supset_{1} \subset_1$
if  $c_{s+1}=\vee $  and
$k$ is even (resp. odd);

\item[---]  $c'_0= \supset_1$ (resp. $\supset_2$) and
$c'_{s+1}=\subset_{2}\supset_{2} \subset_2$
if   $c_{s+1}=\land$  and
$k$ is even (resp. odd).
\end{itemize}
Finally replace each $\bullet_j$ with $\subset_{j}\supset_{j}$ and
each $\times_j$ with $\supset_{j}\subset_{j}$ in $c'_0c_1c_2\ldots
c'_{s+1}$, and do the following final substitutions:
\begin{itemize}
\item[---] replace each $\supset_j\subset_j$ by $\sigma_j^{-1}$ (see Figure
  \ref{fig:curve to braid}a);
\item[---] replace each $\supset_1\subset_2$ by $\sigma_1^{-1}\sigma_2^{-1}\sigma_1$ (see Figure
  \ref{fig:curve to braid}b);
\item[---] replace each $\supset_2\subset_1$ by $\sigma_2^{-1}\sigma_1^{-1}\sigma_2$ (see Figure
  \ref{fig:curve to braid}c).
\end{itemize}
\begin{figure}[h]
\begin{center}
\begin{tabular}{ccccc}
\includegraphics[height=1cm, angle=0]{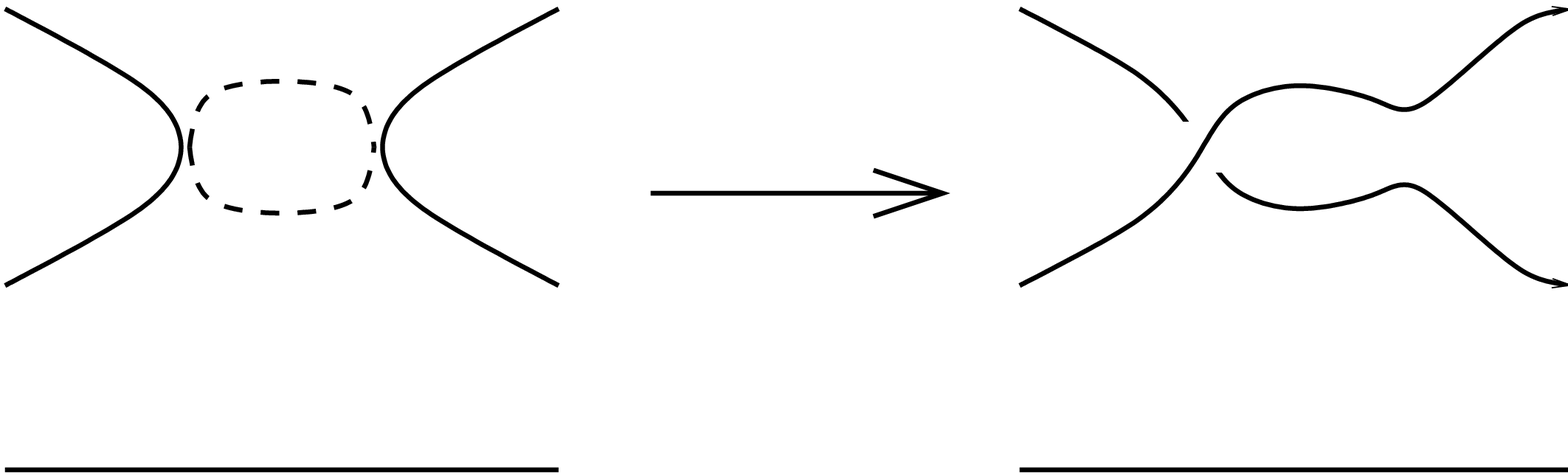}& \vspace{3ex} &
\includegraphics[height=1cm, angle=0]{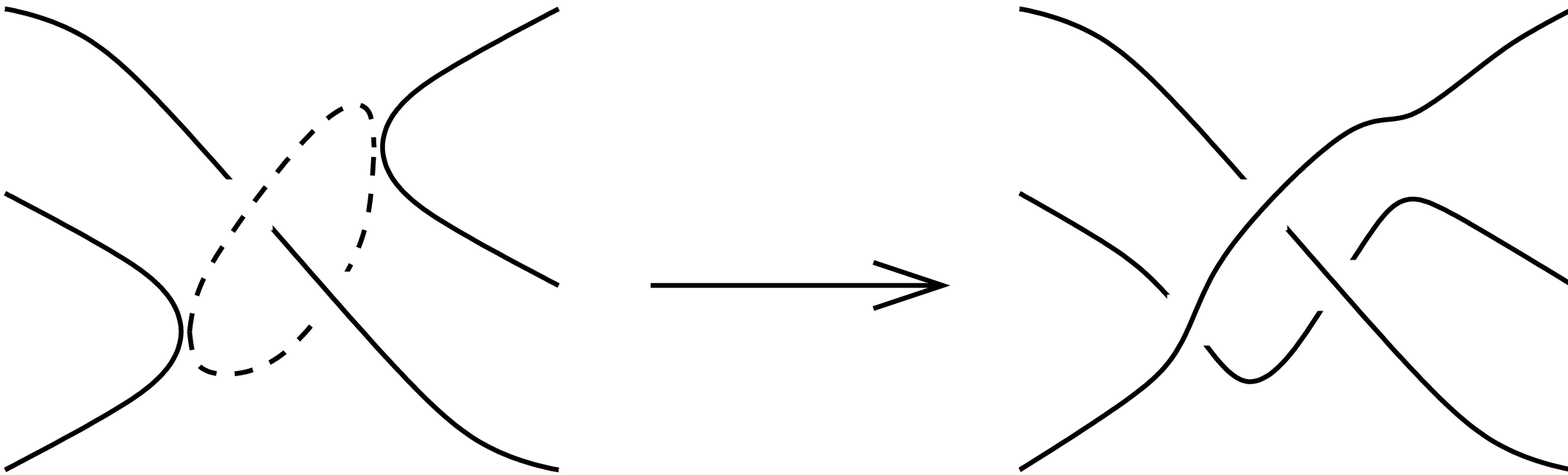}& \vspace{3ex} &
\includegraphics[height=1cm, angle=0]{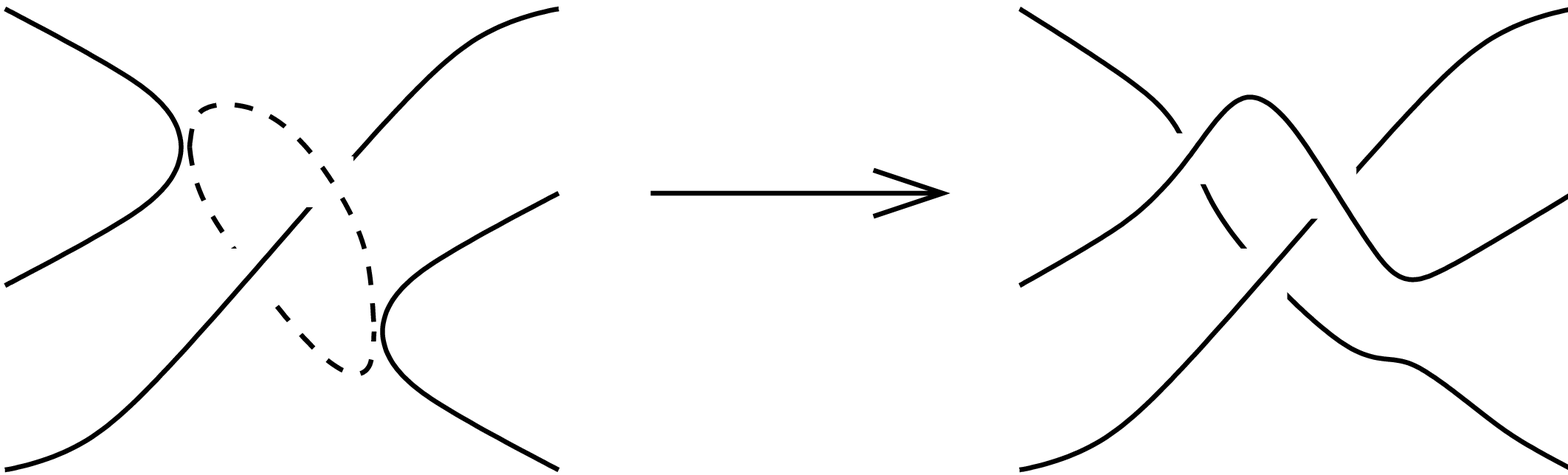}
\\ a) $\supset_j\subset_j \to \sigma_j^{-1}$  && b)
$\supset_1\subset_2 \to \sigma_1^{-1}\sigma_2^{-1}\sigma_1$&& c)
$\supset_2\subset_1 \to \sigma_2^{-1}\sigma_1^{-1}\sigma_2$
\end{tabular}
\end{center}
\caption{}
\label{fig:curve to braid}
\end{figure}
Then we obtain a braid $b_\mathbb R$. We define the braid associated
 to the $\mathcal L$-scheme $\LL_C$, denoted $b_C$,
as the braid
$b_\mathbb R (\sigma_1\sigma_2\sigma_1)^k$.
\begin{exa}
After the two replacements described above, we obtain in the case of
Example
\ref{ex:running ex}
\begin{center}
\begin{tabular}{cccccccc}
&$\bullet_1$ &$\subset_1$ & $\times_2$ & $\times_1$ & $\supset_1$ & $ \vee$ &\\
$\overbrace{\supset_2}$ & $\overbrace{\subset_1\supset_1}$ & $\overbrace{\subset_1}$ & $
\overbrace{\supset_2\subset_2}$ & $\overbrace{\supset_1\subset_1}$ & $\overbrace{\supset_1}$&
$ \overbrace{\subset_1\supset_1\subset_1}$&\\
\multicolumn{3}{c}{$\sigma_2^{-1} \sigma_1^{-1} \sigma_2 \sigma_1^{-1}$}&$\sigma_2^{-1}$&$\sigma_1^{-1}$&
\multicolumn{2}{c}{$\sigma_1^{-2}$} & $(\sigma_1\sigma_2\sigma_1)^2$
\end{tabular}
\end{center}
which finally gives the braid
$$
b_C=\sigma_2^{-1}\sigma_1^{-1}\sigma_2\sigma_1^{-1}
\sigma_2^{-1}\sigma_1^{-3}(\sigma_1\sigma_2\sigma_1)^2$$
which is the trivial braid, see Figure \ref{Ex2}.
\begin{figure}[h]
\begin{center}
\begin{tabular}{c}
\includegraphics[height=1cm, angle=0]{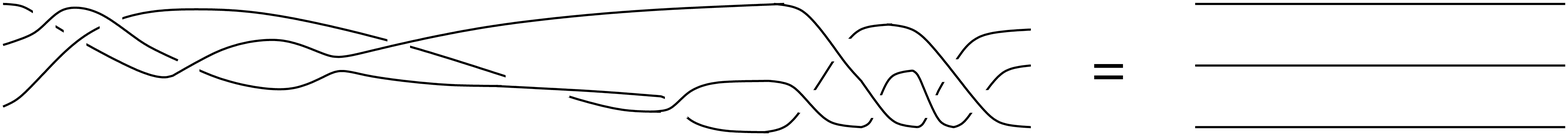}
\\
\\ The braid $b_C$ of the parametrization given in  Example \ref{ex:running ex}
\end{tabular}
\end{center}
\caption{}
\label{Ex2}
\end{figure}
\end{exa}
\begin{prop}
The closure of the braid $b_C$ is the oriented link $L_C$ associated to
$C$.
\end{prop}
\begin{proof}
The only part which is not already contained in
{\cite[Section 3]{O1}} is the determination of the braid
$C\cap \pi^{-1}(\gamma)$ where $\gamma$ is the arc in $\CC$
parametrized by $u\in [0;\pi]\to \varepsilon + R_1 \exp(2i\pi u)$
(recall that $\varepsilon$ and $R_1$ have been introduced at the end
of Section \ref{braid}).
But this braid is dictated by the  monomials
of an equation of
$C$  in $\CC^2$ which are dominating at infinity.
 Up to a real multiplicative constant, these monomials
 are precisely
\begin{itemize}
\item[---] $y^3 - \lambda x^{3k-1}$ if $c_{s+1}=\downarrow$,
\item[---] $y^3 + \lambda x^{3k-1}$ if $c_{s+1}=\uparrow$,
\item[---] $y^3 - \lambda x^{3k-2}$ if $c_{s+1}=\vee$,
\item[---] $y^3 + \lambda x^{3k-2}$ if $c_{s+1}=\land$,
\end{itemize}
where $\lambda$ is some positive real number.
Now the result follows from straightforward local computations as in
{\cite[Section~3]{O1}}.
\end{proof}

\section{Proof of Theorem \ref{thm:main}}\label{rph}
Here we apply the method  described in the previous section to
the study of the lexicographic degree of two-bridge knots.
Given a long knot diagram $D$ in $\RR^2$, we denote by $\overline D$
its projection to $\RR^2$ (i.e. we forget about the sign of the crossings).

\subsection{The case of alternating diagrams}\label{sec:special}
We first consider the special case of alternating diagrams, from which
we will deduce the general case in Section \ref{sec:general}.
\begin{prop}\label{prop:lower bound}
Let $D$ be the alternating trigonal diagram
$C(m) $ with $m$ odd, or $C(m,n)$ with  $mn$ even and positive.
If $\gamma:\RR\to\RR^2$ is a polynomial map of bidegree $(3,b)$ such
that $\gamma(\RR)$ is $\LL$-isotopic to $\overline D$,
then $b\geq \pent{3N-1}2$, where $N$ is
the crossing number of the knot.
\end{prop}
\begin{proof}
Up to a change of coordinates in $\CC^2$, we may suppose that
$b=3k-1$ or  $b=3k-2$.
By genericity argument, we may also suppose that $C=\gamma(\CC)$ is a
nodal curve in $\CC^2$. Finally by symmetry we may further
assume that $m$ is
even
when
$m+n$ is odd.

In the case of $C(m)$,  the $\LL$-scheme of $C$ contains the
patterns $\subset_2(\times_1)^{m}\supset_2 $. In the case of
$C(m,n)$, the $\LL$-scheme of $C$ contains the
patterns $\subset_2(\times_1)^{m}(\times_2)^{n}\supset_1 $ (see Figure
\ref{label link}).
\begin{figure}[!th]
\begin{center}
\psfrag{x}{$\scriptstyle 1$}
\psfrag{y}{$\scriptstyle 2$}
\psfrag{z}{$\scriptstyle 3$}
\psfrag{a}{$m$}\psfrag{b}{$n$}
\begin{tabular}{ccc}
{\scalebox{.6}{\includegraphics{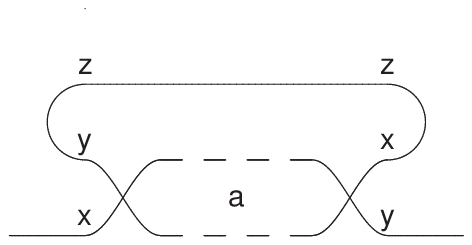}}}&
{\scalebox{.6}{\includegraphics{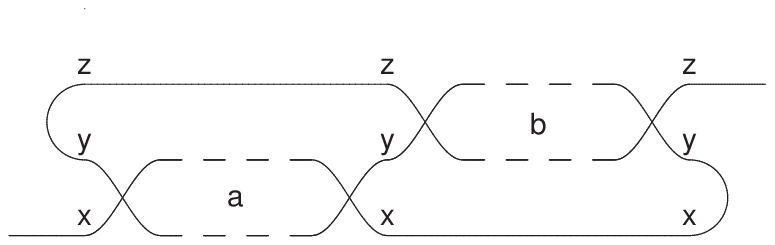}}}&
{\scalebox{.6}{\includegraphics{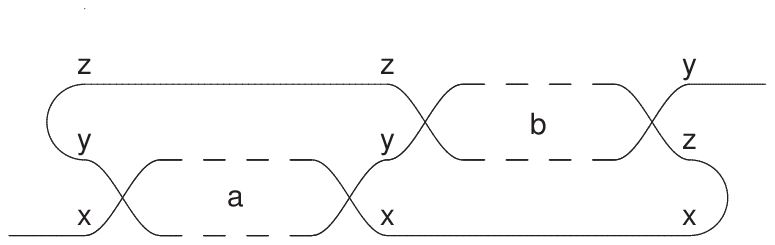}}}\\
$C(m)$&$C(m,n)$ with $n$ even&
$C(m,n)$ with $n$ odd
\end{tabular}
\end{center}
\caption{Diagrams for $C(m)$ or $C(m,n)$ knots}
\label{label link}
\end{figure}
By the Riemann--Hurwitz formula the map $\pi\circ \gamma$ has exactly two
ramification points, which correspond
precisely to the patterns $\subset$ and $\supset$ in the $\LL$-scheme
of $C$.
In particular, these two ramification points are real.
Hence,  still by the Riemann--Hurwitz formula, the surface $U_C$ is made of
three embedded disks $U_1,U_2,$ and $U_3$.
Each disk $U_i$ bounds the subknot $L_i$ of the link $L_C$.
The positivity of self-intersections of a complex algebraic
curve and Proposition \ref{inters} implies that
$\lk(L_i,L_j) = U_i \cdot U_j \geq 0 $.
On the other hand $U_i$ and $U_j$ intersect transversally at
complex conjugated nodes,  thus we have
$$\sum_{1\le i<j\le 3} U_i \cdot U_j= \beta,$$
and therefore
\begin{equation}\label{inter}
\lk(L_i,L_j)\le \beta, \quad 1 \leq i<j \leq 3.
\end{equation}
Moreover Lemma \ref{lem:nodes} states that
\begin{equation}\label{equ genus}
N+\alpha+2\beta= b-1.
\end{equation}
Let us label the components $L_1,L_2$, and $L_3$
as depicted in Figure \ref{label link}, and
let us evaluate the quantity  $2\lk (L_1,L_3)$
from the algorithm described in Section
\ref{sec:closed braid}.
Clearly, no crossing point of $\RR C$ contributes to $2\lk (L_1,L_3)$,
and the factor $(\sigma_1\sigma_2\sigma_1)^k$ contributes  $k$ to
$2\lk (L_1,L_3)$.
As in Section \ref{sec:closed braid}, we denote by $c_1\ldots c_{s+1}$ the
$\mathcal L$-scheme realized by $C$, and by $c'_0c_1\ldots
c_sc'_{s+1}$ the result of replacing $c_{s+1}$ by
$c'_0$ and $c'_{s+1}$.
We have $N=n+m$ and $s=N+\alpha+2$.
Let $p\ge 1$ such that $c_p=\subset_2$. Then we have
$c_{p+N+1}=\supset_2$ and
$c_q=\bullet_j$ for $1\le q \le p-1$ and $N+p+2\le q\le s$.
We can estimate the contribution of $b_\RR$ to $2\lk (L_1,L_3)$ as
follows:
\begin{enumerate}
\item[---] in the case of $C(m)$ or $C(m,n)$ with $N$ odd:

\noindent each pattern $c_q$
with $p+N+1\le q\le s$ contributes  $-1$;
each pattern
$c_qc_{q+1}$ with $1\le q\le p-1$
 contributes at least $-1$;
the pattern
$c'_0 c_1$ contributes at least $0$;
the pattern $c'_{s+1}$ contributes  $0$ if $b=3k-1$, and  $-1$ if $b=3k-2$.
\item[---] in the case of $C(m,n)$ with $N$ even:

\noindent
each pattern
$c_qc_{q+1}$ with $1\le q\le p-1$ or $p+N+1\le q\le s-1$
 contributes at least $-1$, as well as
the pattern
$c'_0 c_1$;
the pattern $c_sc'_{s+1}$ contributes at least $-1$ if $b=3k-1$, and at
least $0$ if $b=3k-2$.
\end{enumerate}

Altogether, writing $d=3k-1-\varepsilon$, we obtain
$$2\lk(L_1,L_3)\ge k-\alpha-1-\varepsilon $$
if $N$ is odd, and
$$2\lk(L_1,L_3)\ge k-\alpha-2+\varepsilon $$
if $N$ is even.
These last inequalities together with identities $(\ref{inter})$ and
$(\ref{equ genus})$ give
$$b-k\ge N-\varepsilon $$
if $N$ is odd, and
$$b-k\ge N-1+\varepsilon $$
if $N$ is even.
Hence we obtain
$$b\ge\pent{3N-1}2$$ which proves the proposition.
\end{proof}
\begin{rem}
The proof of Proposition \ref{prop:lower bound} shows that if
$b=\pent{3N-1}2$ then $\beta=0$  and $\alpha=\frac{N-3}{2}$ or
$\alpha=\frac{N}{2}-2$ depending on the parity of $N$. Moreover
the mutual position of the solitary nodes of
$C$  is imposed: the $\LL$-scheme realized by $C$ does not
contain the pattern  $(\bullet_j)^2$ (i.e. the pattern corresponding to two
successive solitary nodes must be $\bullet_j\bullet_{j\pm 1}$).
\end{rem}

\begin{rem}
Let us consider the harmonic knot $\H(3,7,11)$.
It is proved in \cite{KP1} that its crossing number is $6$.
Then its lexicographic degree is $(3,7,c),$ which is smaller than the
upper bound given
in Proposition \ref{prop:lower bound}.
This is the knot $6_3$ which is neither a torus knot nor a generalized
twist knot (\cite{KP4}).
In Section \ref{easy} we shall see that the lexicographic degree of
$6_3$ is $(3,7,11)$.
\begin{figure}[!ht]
\begin{center}
{\scalebox{.8}{\includegraphics{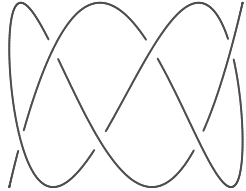}}}
\end{center}
\caption{The harmonic knot $ \H(3,7,11)= 6_3$}
\label{ctrex}
\end{figure}
\end{rem}

\subsection{Reduction to the case of alternating diagrams: real
  pseudoholomorphic curves}\label{sec:general}
As explained in the introduction
it might not be true that the lexicographic degree of a trigonal long knot diagram
does not increases along an isotopy through trigonal diagrams
which never increases the number of crossings. For this to be
true we have to consider \emph{real pseudoholomorphic curves} (see
Corollary \ref{cor:isotopy}).

We consider $\CC^2$ equipped with the standard Fubini--Study symplectic
form
$$\omega(u+iv,u'+iv' )=u\land v + u'\land v'. $$
Recall that an \emph{almost complex structure}
 $J$ on $\CC^2$ is  a smooth section
$J: \CC^2\to End_\RR(T\CC^2)$ such that $J_p^{2}=-Id_{T_p\CC^2}$ for any $p\in\CC^2$, and that $J$ is
said to be \emph{tamed} by $\omega$ if $\omega(v,Jv)>0$ for any
non-null vector $v\in  T\CC^2$. Such an almost complex structure
is called \emph{real} if the standard complex conjugation $conj$ on
$\CC^2$ is $J$-antiholomorphic  (i.e. $conj\circ J=J^{-1}\circ conj$).
For example, the standard complex structure $J_0$
on $\CC^2$ is real and tamed by $\omega$.
An almost complex structure
is called
 \emph{admissible} if it is real, tamed by $\omega$, and
coincide with
$J_0$ outside a compact set in $\CC^2$.

The next definition follows the lines of {\cite[Section 1]{OS1}} and
{\cite[Chapter 2]{McS}}.
We also denote by $J_0$ (resp. $conj$) the standard
complex structure
(resp. complex conjugation) on $\CC$.
\begin{defi}
  Let $J$ be an admissible  almost complex structure on $\CC^2$. A
\emph{real (rational)  $J$-holomorphic curve} is a map
$\gamma:\CC \to\CC^2$ which is polynomial outside a compact set of $\CC$, and such that
$$d\gamma'\circ J_0=J\circ d\gamma' \quad\mbox{and}\quad \gamma\circ
  conj=conj\circ\gamma. $$

A \emph{real pseudoholomorphic curve} $\gamma:\CC \to\CC^2$ is a map
which is $J$-holomorphic for some  admissible
 almost complex structure
$J$
on $\CC^2$.
\end{defi}

\begin{exa}
Every
polynomial real map $\gamma:\RR\to\RR^2$
is a real $J_0$-holomorphic curve. However
not every
 real pseudoholomorphic curve   is $J_0$-holomorphic.
\end{exa}

By definition, the image $C=\gamma(\CC)$ of
any real pseudoholomorphic curve $\gamma$
coincide with a real algebraic curve $C_0$ outside a compact set of $\CC^2$.
We say that  $\gamma$
is of
\emph{bidegree} $(3,b)$ if the curve $C_0$ has for Newton polygon the
triangle with vertices $(0,0),(0,3)$, and $(0,b)$.
As in the algebraic case,
 we may assume for our purposes  that if a real pseudoholomorphic curve is of
bidegree $(3,b)$, then $b=3k-1$ or $b=3k-2$.
Note that nodal pseudoholomorphic curves satisfy the adjunction
formula (see {\cite[Chapter 2]{McS}}), in particular Lemma
\ref{lem:nodes} still holds for pseudoholomorphic curves.

The braid theoretical approach described in Section \ref{rph}
extends word by word
to the case of nodal real
pseudoholomorphic curves of bidegree $(3,b)$ (see for example
\cite{O1,OS1}).
Indeed, in these latter two sections we only use the behavior at infinity of
real algebraic curves and the three following facts:
\begin{enumerate}
\item any intersection of two complex algebraic curves is positive;
\item there exists a unique vertical line passing through a given
  point $p$ of $\CC^2$, which is real if $p$ is real;
\item the map $\pi$ is a real holomorphic map.
\end{enumerate}

Let us fix an admissible almost complex structure  $J$ on $\CC^2$.
The word by word extension of Sections \ref{braid} and
\ref{sec:closed braid} to $J$-holomorphic curves
follows from the three following facts:
\begin{enumerate}
\item[(1')] any intersection of two $J$-holomorphic curves is positive (see
  {\cite[Appendix E]{McS}});
\item[(2')] there is an analogue of the pencil of vertical lines in
  $\CC^2$: consider the standard embedding
$$\begin{array}{ccc}
\CC^2 & \longrightarrow& \CC  P^2
\\ (x,y) &\longmapsto & [x:y:1]
\end{array};$$
 then given a point $p\in\CC^2$
there exists a unique $J$-holomorphic line in $\CC P^2$
passing through $p$ and $[0:1:0]$ (see \cite{Gro}); by uniqueness, this
line is real if $p$ is real; this produces a
real pencil of
``vertical'' $J$-holomorphic lines on $\CC^2$;

\item[(3')] since $J$ is standard outside a compact set of $\CC^2$,
the line $\{y=b\}$ is $J$-holomorphic for $|b|$ large enough; replace
the map $\pi$ by the map $\pi_b:\CC^2\to\{y=b\}$ which associates to
each point $p$ of $\CC^2$ the (unique) intersection point of $\{y=b\}$
with the $J$-holomorphic vertical line passing through $p$; the map
$\pi_b$ is clearly a real $J$-holomorphic map.
\end{enumerate}

The main advantage in dealing with real pseudoholomorphic curves rather
than real algebraic curves is that we have the following proposition.

\begin{prop}[Orevkov, \cite{O1}]\label{operations}
Let $\gamma:\CC\to\CC^2$
be a real nodal pseudoholomorphic curve of bidegree $(3,b)$
realizing an  $\mathcal L$-scheme $\LL_\gamma$, and let $\LL_\gamma'$ be the
$\mathcal L$-scheme  obtained
from $\LL_\gamma$ by one of the following elementary operations~:
$$\begin{array}{c c c c c }
\times_j\supset_{j\pm 1}\ \leftrightarrow\ \times_{j\pm 1}\supset_j &  &
\subset_{j\pm 1}\times_j\ \leftrightarrow\ \subset_j\times_{j\pm 1} & &
 \times_{j}\times_{j}\ \to\ \emptyset
\end{array}$$

$$\begin{array}{ccccc}
\times_j\supset_j\ \leftrightarrow\ \supset_j\bullet_j&  &
\subset_j\times_j\ \leftrightarrow\ \bullet_j\subset_j & &
 \times_{j+ 1}\times_{j}\times_{j+ 1}\ \leftrightarrow \ \times_{j}\times_{j+1}\times_{j}
\end{array}$$
Then there exists a real nodal pseudoholomorphic curve of bidegree $(3,b)$
realizing  $\LL_{\gamma}'$.
\end{prop}
\begin{cor}\label{cor:isotopy}
Let $D_1$ and $D_2$ be two trigonal long knot diagrams such that $D_2$ is
obtained from $D_1$
by an isotopy  of trigonal
diagrams which never increase the number of crossings.
If
there exists a real nodal pseudoholomorphic curve $\gamma_1:\CC\to\CC^2$
of bidegree $(3,b)$ such that $\gamma_1(\RR)$ is
$\LL$-isotopic to $\overline D_1$, then there also
exists a real nodal pseudoholomorphic curve $\gamma_2:\CC\to\CC^2$
of bidegree $(3,b)$ such that $\gamma_2(\RR)$ is
$\LL$-isotopic to $\overline D_2$.
\end{cor}

\begin{prop}\label{prop:lower}
Let $K$ be the two-bridge knot $C(m)$ with $m$ odd, or $C(m,n)$ with $mn$ positive and even, and let
$N$ be  its crossing number.
Let $D$ be any trigonal diagram of $K.$
If there exists a real nodal pseudoholomorphic curve $\gamma:\CC\to\CC^2$
of bidegree $(3,b)$ such that $\gamma(\RR)$ is
$\LL$-isotopic to $\overline D$,
then $b\ge \pent{3N-1}2$.
\end{prop}
\Pf
According to Corollary \ref{cor:isotopy} and Theorem
A,
it is enough to prove the proposition when $D$ is
alternating, which has been done in Proposition \ref{prop:lower bound}
(recall that real pseudoholomorphic curves satisfy Lemma
\ref{lem:nodes}, so the proof of Proposition \ref{prop:lower bound} extends
word by word to this case).
\EPf

\begin{thm}\label{thm:main}
Let $K$ be the two-bridge knot $C(m)$ with $m$ odd, or $C(m,n)$ with $mn$ positive and even, and let
$N$ be  its crossing number. Then the  lexicographic degree of $K$ is
$$ \Bigl(3, \pent{3N-1}2, \pent{3N}2+1 \Bigr).$$
\end{thm}
\Pf
Let $K$ be any two-bridge knot. It is proved in \cite{KP4} that there
exists a polynomial curve
$\gamma:\RR\to\RR^3$
of degree $(3,b,c)$ such that $N<b<c$, $b+c=3N$, and
$\gamma(\RR)$ is the knot $K$. In particular we have
$ b\leq \pent{3N-1}2$ and $ c\geq \pent{3N}2+1$.
If $K$ is $C(m)$  or $C(m,n)$, then
by Proposition \ref{prop:lower} we have $b\geq \pent{3N-1}2$.
This  implies that $ c\leq \pent{3N}2+1$, and the theorem is proved.
\EPf

\begin{rem} It is shown in \cite{KP3}  that the harmonic knot
$\H (3,3n+2,3n+1)$ is the
torus knot $ \T(2, 2n+1)=C(2n+1)$.
By Theorem \ref{thm:main}, this is an explicit polynomial parametrization
of minimal lexicographic degree $(3, 3n+1, 3n+2).$
In this example, the number of crossings of the diagram is
greater than the crossing number of the knot.
In \cite{KP2} it is proved that there exists a polynomial representation of
$\T(2,2n+1)$ which is of minimal degree,
and such that the number of crossings is minimal.
\end{rem}

\section{A lower bound }\label{easy}
 Theorem \ref{thm:least} provides
 a general lower bound for the
lexicographic degree of every non-trivial knot, furthermore this lower
bound is sharp.

\begin{prop}\label{prop:alternating}
Let $\gamma:\RR\to\RR^3$ be a polynomial map
of degree $ (a,b,c)$ whose image is a smooth knot.
Suppose that $a$ and $b$ are coprime, and that  the $xy$-diagram of $\gamma(t)$
is alternating  with $ (a-1)(b-1)/2$ crossings.
Then  $ c \ge ab-a-b.$
\end{prop}

The proof of Proposition \ref{prop:alternating} uses
the following lemma of Frobenius type.
\begin{lemma}\label{lemma:frobenius}
Let $a$ and $b$ be coprime positive integers. The number of integers of the form
$ n= \alpha a + \beta b, n \le ab-a-b-1 ,  \ ( \alpha , \beta ) \in
\ZZ_{\ge 0}^2,$
   is equal
to $ (a-1)(b-1)/2.$
\end{lemma}
\begin{proof}
Let  $ N= (a-1)(b-1)/2$. First, let us show that the integer
$2N-1=ab-a-b$ is not of the form
$ \alpha a + \beta b, $ with $ (\alpha, \beta) \in \ZZ_{\ge 0}^2.$

If we had $ ab-a-b= \alpha a + \beta b,$ then $ \alpha \equiv -1  \Mod b  ,$ and
$ \alpha \ge b-1.$ Hence we deduce that $ \alpha a + \beta b  \ge ab-a +\beta b >
2N-1$,  contrary to our hypothesis

Now consider the sets
$E = \{ (\alpha, \beta ) \in \ZZ_{\ge 0} ^2 , \alpha a + \beta b \le 2N-2 \}$
 and
$ Q= \{ (\alpha, \beta) \in \ZZ_{\ge 0}^2, \  \alpha \le b-2, \, \beta \le a-2\}.$
The  set $Q$ has $ 2N$ elements, and  contains $E$.
Denoting
$s( \alpha, \beta) =\alpha a + \beta b ,
\ \alpha' = b-2- \alpha, \,$ and  $\beta ' = a-2- \beta$,   we have
$ s(\alpha , \beta) + s ( \alpha ', \beta ') = 2(2N-1),$ and then
$  ( \alpha', \beta' ) \in Q - E$ if and only if $ (\alpha , \beta ) \in E.$

Consequently, half  the elements of $Q$ are in $E$, and then   $E$ has $N$ elements.
   Since $a$ and $b$ are coprime,
the mapping $s$ is an injection from $E$ to $\ZZ_{\ge 0}$,
 which concludes the proof.
\end{proof}
\begin{proof}[Proof of Proposition \ref{prop:alternating}]
The $N= (a-1)(b-1)/2$ crossing points of the $xy$-projection of
$\gamma(t)=(x(t),y(t),z(t))$
correspond to parameter pairs $(s_i,t_i)$ such that
$$
 x(t_i)=x(s_i),   \    \  y(t_i)=y(s_i), \  \   z(t_i) > z (s_i).
$$

Consider $V$ the vector subspace of $\RR [t]$ generated by the polynomials
$x(t)^{\alpha} y(t)^{\beta}$,
$\alpha a + \beta b  \le 2N-2$.
By  Lemma \ref{lemma:frobenius}, the dimension of $V$ is $N$.
Let us define the linear mapping
$ \varphi: \, V \rightarrow \RR^N$ by
$ \varphi ( h) = \Bigl( h(t_1), \ldots , h(t_N) \Bigr).$
If a polynomial $h$ is in the kernel of $ \varphi$, then
$ h(t_1)=h(s_1)= \cdots = h(t_N)= h(s_N)=0,$ and the polynomial $h$ has $2N$ distinct roots.
Since $\deg (h )\le 2N-2,$ we deduce that $h=0.$

Consequently the mapping $\varphi$ is an isomorphism and there exists
$h \in V$ such that $ h(t_i)=h(s_i)= (z(t_i)+ z(s_i))/2.$

Let us define $ \tilde z (t) = z(t) -h(t) .$
The images of $\gamma$ and $ \tilde \gamma (t) = \bigl(x(t), y(t), \tilde z(t)
\bigr) $
are isotopic
and realize  the same diagram.
By construction,  at each crossing we have $ \tilde z(t_i)> 0 > \tilde z(s_i).$

Let us write $\{t_1, \ldots , t_N, s_1, \ldots, s_N\} = \{\tau_1, \ldots, \tau_{2N}\}$ where
$\tau_j < \tau_{j+1}$.
Since the
diagram realized by $ \tilde\gamma(t)$
is alternating, we may assume that $(-1)^i \tilde z(\tau_i)>0$.
Consequently
 the polynomial
$ \tilde z (t) $ has at least one root in the interval $ ( \tau _j,
 \tau _{j+1} )$, and
$ \deg  (\tilde z) \ge 2N-1 .$
Since $\deg (h) \le 2N-2 ,$  we conclude that $ \deg (z) \ge 2N-1 .$
\end{proof}
As a consequence we deduce:
\begin{thm} \label{thm:least}
Let $K$ be a knot of crossing number $N \ne 0.$
Then the  lexicographic degree of $K$ is at least $(3, N+1, 2N-1).$
For every integer $ N\not\equiv 2 \, \Mod 3, N \ge 3,$ there exists a knot of crossing number $N$
and  lexicographic degree $ (3, N+1, 2N-1)$
\end{thm}
\begin{proof}
Let $(a,b,c)$ be the  lexicographic degree of $K$.
Since $K$ is nontrivial, we have $ a \ge 3$. Suppose that $ a=3$. Then $K$ is a
two-bridge knot, and $b$ is not divisible by 3. Moreover
 by B\'ezout's theorem we have
$N \le (a-1)(b-1)/2= b-1$, that is $b \ge N+1.$

Let $\gamma(t):\RR\to\RR^3 $ be a polynomial parametrization of $K$ of
degree $(3,b,c)$.
If $b= N+1$, then the $xy$-diagram of $\gamma(t)$ has the minimal number of crossings.
Since $K$ is a two-bridge knot, it is an alternating knot, and the
$xy$-diagram of $\gamma(t)$ is alternating.
By Proposition \ref{prop:alternating}, we conclude that  $c\ge 2N-1.$
\pn
On the other hand, when $ N \not\equiv 2 \, \Mod 3 ,$
the harmonic knot $\H(3,N+1, 2N-1)$ has
crossing number  $N$ (see \cite{KP1}), hence
  its  lexicographic degree is $ (3, N+1, 2N-1)$.
\end{proof}
We also obtain the  lexicographic degree of an infinite family of three-bridge knots.
\begin{cor}
Let $\H$ be the harmonic knot $ \H= \H(5,b,4b-5)$, with  $b \not\equiv 0 \,  \Mod 5  .$
The knot $\H$ is a three-bridge knot of  lexicographic degree $ (5,b,4b-5).$
\end{cor}

\begin{proof} The  $xy$-projection of the knot $\H$ has $N=2 (b-1)$
  crossings and is alternating (see \cite{KP3}).
If $\H$ was two-bridged, then it would have a minimal diagram which would be alternating and two-bridged. By Tait's flyping conjecture, this is not the case since the
$xy$-projection of $\H$ is alternating and three-bridged.
\pn
Let $ \bigl( x(t),y(t), z(t) \bigr) $ be a polynomial knot isotopic to $\H.$
Since $\H$ is not two-bridged, we have $ \deg (x) \ge 5. $
If $\deg (x) =5,$ then  by B\'ezout's theorem we have $ N \le 2 \bigl( \deg (y) -1  \bigr) ,$ that is
$  \deg (y) \ge b.$
If $ \deg (y)= b,$ then by Proposition \ref{prop:alternating} we obtain $\deg (z) \ge 5b-5-b= 4b-5.$
\end{proof}

\begin{remark}
The  lexicographic degree of the harmonic knot $\H(a,b,ab-a-b)$ is not always
 $ (a, b, ab-a-b).$
For example, if $a=4$  this knot is two-bridged and has a polynomial parametrization of degree
$(3,b',c'), $ which is lexicographically smaller than $( 4, b, 3b-4).$
\end{remark}

Proposition \ref{prop:alternating} can also be used to study the maximal degree of a polynomial knot, that is
the maximum degree of its components.

\begin{cor}
Let $K$ be a polynomial knot of degree $d$ and crossing number $N.$
We have
$$
N \le \Frac{(d-2)(d-3)}{2}.
$$
Moreover, if $d>5$ and if $K$ is alternating, then we have:
$$
N \le \Frac{(d-1)(d-4)}{2} .
$$
\end{cor}

\begin{proof} Using an affine transformation, we can suppose that the lexicographic degree of $K$
is $\le (d-2,d-1,d)$. By B\'ezout's theorem we see that the number of crossings of the $xy$-projection
is at most $ (d-2)(d-3)/2,$ which proves the first assertion.
If we had  $ N= (d-2)(d-3)/2,$ then the lexicographic degree of $K$
would be $(d-2,d-1,d)$.
By Proposition \ref{prop:alternating},
we have $d \geq (d-2)(d-1) - (d-2)-(d-1)= d^2-5d+5$, which is impossible, since $d>5$.
Consequently, we have $ N \le (d-2)(d-3)/2 -1 = (d-1)(d-4)/2.$
\end{proof}

\begin{remark} Since knots of crossing number smaller than 8 are alternating,
we deduce that the only knot of degree 4 is the trivial knot, the only nontrivial knot of degree 5
is the trefoil, and the crossing number of a knot of degree 6 is at most 5.
\end{remark}

\small

\bibliographystyle{alpha}

\goodbreak
\vfill
\pn
\hrule width 7cm height 1pt 
\pn
{\small
Erwan {\sc Brugallé}\\
École Polytechnique,
Centre Mathématiques Laurent Schwartz, 91 128 Palaiseau Cedex, France\\
e-mail: {\tt erwan.brugalle@math.cnrs.fr}
\pn
Pierre-Vincent {\sc Koseleff}\\
Universit\'e Pierre et Marie Curie (UPMC Sorbonne Universit\'es),\\
Institut de Math\'ematiques de Jussieu (IMJ-PRG)  \& Inria-Rocquencourt\\
e-mail: {\tt koseleff@math.jussieu.fr}
\pn
Daniel {\sc Pecker}\\
Universit\'e Pierre et Marie Curie (UPMC Sorbonne Universit\'es),\\
Institut de Math\'ematiques de Jussieu (IMJ-PRG),\\
e-mail: {\tt pecker@math.jussieu.fr}
}

\end{document}